\documentclass[11pt]{article}

\usepackage[margin=1in]{geometry}
\usepackage{amsmath,amssymb,amsthm,mathtools}
\usepackage{enumitem}
\usepackage{microtype}
\usepackage[colorlinks=true,linkcolor=blue,citecolor=blue,urlcolor=blue]{hyperref}
\usepackage{tikz}

\numberwithin{equation}{section}

\newtheorem{theorem}{Theorem}[section]
\newtheorem{lemma}[theorem]{Lemma}
\newtheorem{proposition}[theorem]{Proposition}
\newtheorem{corollary}[theorem]{Corollary}
\newtheorem{remark}[theorem]{Remark}

\newcommand{\E}{\mathbb E}
\newcommand{\Pp}{\mathbb P}
\newcommand{\Var}{\operatorname{Var}}

\newcommand{\dist}{\operatorname{d}}

\title{Nonconcentration of hitting times for random walks on graphs}
\author{Rafael Chiclana}
\date{}

\begin{document}
\maketitle

\begin{abstract}
	We study nonconcentration of hitting times for simple random walk on finite
	graphs. We prove that, for every connected graph with $n$ vertices,
	\[
	\operatorname{Var}_x(\tau_y)+\mathbb E_x\tau_y
	\ge
	\frac{(\mathbb E_x\tau_y)^2}{1+\log n},
	\]
	with the logarithmic term sharp up to constants.  Under a bounded-degree
	assumption the additive mean term can be removed, giving a variance lower bound
	depending only on \(\mathbb E_x\tau_y\) and the graph distance \(\dist(x,y)\).  We show that this
	degree assumption is necessary by constructing high-degree graphs with
	linear mean and bounded variance; the same construction disproves a conjecture of Norris--Peres--Zhai concerning local nonconcentration of hitting times. We also prove a sharper tree estimate, extend the main argument to finite reversible Markov chains, and show that Holroyd's interval conjecture, stated in Norris--Peres--Zhai, fails even for bounded-degree trees.
\end{abstract}

\section{Introduction and main results}

Let \(G=(V,E)\) be a graph, where \(V\) is the set of vertices and \(E\) is the
set of edges.  We say that \(G\) is connected if there is a path between every
pair of vertices, and simple if it has no loops or repeated edges. Throughout
this paper, all graphs are assumed to be finite, connected, and simple.

A simple random walk on \(G\) is the Markov chain \((X_t)_{t\ge0}\) with state
space \(V\) which, at each step, moves from its current vertex to a uniformly
chosen neighboring vertex. When the walk starts from \(x\), we write \(\mathbb P_x\), \(\mathbb E_x\), and
\(\operatorname{Var}_x\) for the corresponding probability, expectation, and
variance. Given a vertex \(y\in V\), the hitting time of \(y\)
is
\[
\tau_y:=\inf\{t\ge0:X_t=y\}.
\]
Hitting times are one of the central objects in the study of random walks on
graphs. They measure how long it takes for the walk to reach a prescribed target
vertex, and they appear naturally in many areas, including network analysis,
theoretical computer science, interacting particle systems, and mathematical
finance; see, for instance,
\cite{Avrachenkov2018hitting,Balka2009review,nadtochiy2019particle}.

Recently, it has been shown that hitting times exhibit nonconcentration properties (see \cite{gurel2013nonconcentration,norris2017surprise}). In this paper, we establish lower bounds for the variance $\operatorname{Var}_x(\tau_y)$ in terms of \(\mathbb E_x\tau_y\) and basic geometric parameters of the graph. Our main result is the following general inequality.

\begin{theorem}\label{thm:general-plus-mean}
	Let \(G\) be a graph with \(n\) vertices. Then, for every
	\(x,y\in V\),
	\begin{equation}\label{eq:main}
		\operatorname{Var}_x(\tau_y)+\mathbb E_x\tau_y
		\ge
		\frac{(\mathbb E_x\tau_y)^2}{1+\log n}.
	\end{equation}
\end{theorem}

A family of trees constructed by Norris, Peres, and Zhai in the proof of \cite[Claim 2.2]{norris2017surprise}, with expectation and variance recorded in \eqref{eq:peres}, shows that the logarithmic term in \eqref{eq:main} is sharp up to constants.

We write \(\deg(v)\) for the degree of a vertex \(v\), and \(\dist(x,y)\) for
the graph distance between \(x\) and \(y\).

\begin{remark}\label{rem:remarkmain}
The parameter \(n\) in Theorem~\ref{thm:general-plus-mean} is only a coarse
global bound.  The same argument gives the following localized estimate: for
\(x\neq y\),
\begin{equation}\label{eq:mainbounded}
	\Var_x(\tau_y)+\E_x\tau_y
	\ge
	\frac{2(\E_x\tau_y)^2}
	{2+\log \dist(x,y)+\log\deg(y)}.
\end{equation}
This implies \eqref{eq:main}, since \(\dist(x,y)\le n\) and \(\deg(y)\le n\).
\end{remark}

For bounded-degree graphs, the variance itself dominates the mean.  Thus, from Remark \ref{rem:remarkmain} we deduce a genuine variance lower bound.

\begin{corollary}\label{thm:bounded-degree}
	Let \(G\) be a graph with maximum degree at most \(D\). Assume
	that \(\tau_y\) is not deterministic. Then
	\[
	\operatorname{Var}_x(\tau_y)
	\ge
	c_D
	\frac{(\mathbb E_x\tau_y)^2}{\log(\dist(x,y)+1)},
	\qquad \mbox{where } c_D=\frac{1}{6D(2+\log D)}.
	\]
\end{corollary}

The bounded-degree assumption cannot be removed. In section \ref{sec:funnel}, we construct a family of high-degree graphs for which the variance lower bound fails.

\begin{proposition}\label{prop:funnel}
	For every \(L\ge2\), there exists a graph
	\(G_L=(V_L,E_L)\) and vertices \(x,y\in V_L\) such that
	\[
	\mathbb E_x\tau_y\asymp L,
	\qquad
	\operatorname{Var}_x(\tau_y)=O(1),
	\qquad
	\log |V_L|\asymp L\log L.
	\]
	Consequently,
	\[
	\frac{(\mathbb E_x\tau_y)^2}{1+\log |V_L|}
	\asymp
	\frac{L}{\log L}
	\to\infty,
	\qquad
	\operatorname{Var}_x(\tau_y)=O(1).
	\]
\end{proposition}

The construction is also related to a conjecture of Norris, Peres, and Zhai
\cite[Conjecture~6.1]{norris2017surprise}. They proved in \cite{norris2017surprise} that, for simple random walk on an
\(n\)-vertex graph,
\[
\mathbb P_x(\tau_y=t)\le C\frac{\log n}{t},
\]
and conjectured that the logarithmic factor could be replaced by a square-root
logarithmic factor; that is, that one has
\[
\mathbb P_x(\tau_y=t)\le C\frac{\sqrt{\log n}}{t}
\]
for some universal constant $C>0$. By a standard Chebyshev argument, such an estimate
would imply a variance lower bound of the form
\[
\operatorname{Var}_x(\tau_y)
\ge
c\frac{(\mathbb E_x\tau_y)^2}{1+\log n}.
\]
Proposition~\ref{prop:funnel} shows that this consequence is false in general,
and hence disproves the unrestricted form of the conjecture.  Whether the
square-root logarithmic bound holds under a bounded-degree assumption remains
open; see Remark~\ref{rem:local-counterexample} for the direct local
calculation.

The mechanism behind Corollary~\ref{thm:bounded-degree} is not special to bounded
degree.  The additive mean term in \eqref{eq:main} can be removed whenever one
has an independent estimate of the form
\[
\operatorname{Var}_x(\tau_y)\ge c\,\mathbb E_x\tau_y.
\]
Trees provide an important example. Recall that a tree is a connected graph with no cycles. For trees, the variance dominates the mean without any bounded-degree assumption, and the logarithmic term depends only on the distance between the starting point and the target.

\begin{corollary}\label{theo:tree}
	Let \(T\) be a tree. Then, for any vertices \(x,y\in T\) such that \(\tau_y\) is
	not deterministic,
	\[
	\operatorname{Var}_x(\tau_y)
	\ge
	\frac{\log 4}{3}\,
	\frac{(\mathbb E_x\tau_y)^2}{\log(\dist(x,y)+1)}.
	\]
\end{corollary}

The argument behind Theorem \ref{thm:general-plus-mean} is naturally electrical and extends beyond simple random walk.  If \(P\) is the transition matrix of a finite reversible Markov chain with stationary measure
\(\pi\), then the conductances
\[
c(u,v):=\pi(u)P(u,v)=\pi(v)P(v,u)
\]
turn the chain into a weighted network. In Appendix \ref{app:reversible} we present Theorem \ref{theo:reversible}, a generalization of Theorem~\ref{thm:general-plus-mean} for reversible Markov chains. For clarity, we keep the main text focused on
simple random walk on graphs; analogous weighted versions of the auxiliary
estimates can be obtained in the same way.

Finally, we also revisit Holroyd's interval conjecture, recorded by
Norris, Peres, and Zhai as \cite[Conjecture~6.2]{norris2017surprise}. This
conjecture concerns intervals rather than single times, and states that, for a
Markov chain with \(n\) states,
\[
\mathbb P_x(t\le \tau_y\le t+n)\le C\frac{n}{t},
\qquad t>n.
\]
In Appendix~\ref{app:interval-conjecture}, we show that this interval conjecture
is false even for simple random walk on bounded-degree trees.  The counterexample
uses the trees constructed by Norris, Peres, and Zhai in the proof of
\cite[Claim~2.2]{norris2017surprise}.

The rest of the paper is organized as follows. In section \ref{sec:flow-notation}, we introduce notation and preliminary results used throughout the paper. In Section \ref{sec:3}, we prove our main result Theorem \ref{thm:general-plus-mean}. The lemmas required in this proof are proved in Sections \ref{sec:4} and \ref{sec:5}. In Section \ref{sec:6}, we prove Corollary \ref{thm:bounded-degree}. The counterexample provided by Proposition \ref{prop:funnel} is constructed in Section \ref{sec:funnel}. Section \ref{sec:8} contains the proof of Corollary \ref{theo:tree}. Finally, in Appendix \ref{app:reversible} we generalize our main result to reversible chains, and in Appendix \ref{app:interval-conjecture} we provide a counterexample to Holroyd's interval conjecture \cite[Conjecture~6.2]{norris2017surprise}.

\section{Notation and preliminaries}\label{sec:flow-notation}

All logarithms are natural. We use standard asymptotic notation: \(A_L=O(B_L)\) means \(A_L\le C B_L\) for a universal constant \(C\), and \(A_L\asymp B_L\) means both \(A_L=O(B_L)\) and
\(B_L=O(A_L)\). For vertices $u$, $v$ of a graph, we write \(u\sim v\) when \(u\) and \(v\) are adjacent. If \(\theta\) is an antisymmetric function on oriented edges,
its divergence is
\[
\operatorname{div}\theta(u):=\sum_{v\sim u}\theta(u,v).
\]
We say that \(\theta\) is a unit flow from \(x\) to \(y\) if
\[
\operatorname{div}\theta=\mathbf 1_{\{x\}}-\mathbf 1_{\{y\}}.
\]
When $\theta$ is a unit flow from $x$ to $y$, we will use the summation-by-parts identity
\begin{equation}\label{eq:summation-by-parts}
	\sum_{\{u,v\}\in E}\theta(u,v)\bigl(f(u)-f(v)\bigr)
	=
	\sum_{u\in V}f(u)\operatorname{div}\theta(u) = f(x)-f(y),
\end{equation}
where the expression on the left is independent of the chosen orientation of
each edge. When a function \(g:V\to\mathbb R\) satisfies \(g(y)=0\) and \(\theta_g(u,v):=g(u)-g(v)\) is a unit flow from \(x\) to \(y\), we call \(g\) the unit-current voltage from \(x\)
to \(y\). The effective resistance between vertices $x$ and $y$ is defined by
\[
R_{\mathrm{eff}}(x,y):=g(x)-g(y)=g(x).
\]
A standard inequality we will use is $R_{\mathrm{eff}}(x,y)\le \dist(x,y)$, which follows from Thomson's principle. For a detailed treatment of electrical networks, we refer to \cite{levin2017markov}. Finally, in sums of the form
\[
\sum_{\{u,v\}\in E}\frac{(g(u)-g(v))^2}{g(u)+g(v)},
\]
edges with \(g(u)=g(v)=0\) are understood to contribute \(0\).

\section{Proof of Theorem \ref{thm:general-plus-mean}}\label{sec:3}
Fix \(x,y\in V\). If \(x=y\), then \(\tau_y=0\), so there is nothing to prove.
Thus we assume \(x\neq y\). We begin by introducing the two functions that will appear
throughout the argument. First, let $h \colon V \longrightarrow \mathbb{R}$ be defined by
\[
h(v):=\E_v\tau_y \quad \forall \, v \in V.
\]
Thus \(h(v)\) is the expected time to hit \(y\) when the walk starts from \(v\). Next, let \(\mu\) be the occupation measure of the walk before it hits \(y\), that is,
\[
\mu(v):=\E_x\#\{0\le t<\tau_y:X_t=v\},
\]
and define $g \colon V \longrightarrow \mathbb{R}$ by
\[
g(v):=\frac{\mu(v)}{\deg(v)} \quad \forall \, v \in V.
\]
Since the occupation is counted only before \(\tau_y\), we have \(g(y)=0\).

Equivalently, \(g(u)\) is the expected number of crossings of any fixed
oriented edge leaving \(u\) before time \(\tau_y\). Hence the antisymmetric edge function
\[
\theta_g(u,v):=g(u)-g(v)
\]
represents the expected net flow across the oriented edge \((u,v)\).  We claim
that \(\theta_g\) is a unit flow from \(x\) to \(y\). Indeed, for any stopped path from $x$ to $y$, the net number of exits from a vertex $u$ minus the number of entrances into $u$ is 
\[\mathbf 1_{\{u=x\}}-\mathbf 1_{\{u=y\}}.\]
Taking expectations gives
\begin{equation}\label{eq:g-unit-current}
	\operatorname{div}\theta_g(u)
	=
	\sum_{v\sim u}\bigl(g(u)-g(v)\bigr)
	=
	\mathbf 1_{\{u=x\}}-\mathbf 1_{\{u=y\}}.
\end{equation}
This observation, together with the summation-by-parts identity \eqref{eq:summation-by-parts},
allows us to express the mean hitting time in terms of the flow
\(\theta_g\) and \(h\).

\begin{lemma}\label{lem:mean-current-identity}
	For every \(x,y\in V\),
	\begin{equation}\label{eq:mean-current-identity}
		\E_x\tau_y
		=
		\sum_{\{u,v\}\in E}
		\bigl(g(u)-g(v)\bigr)\bigl(h(u)-h(v)\bigr).
	\end{equation}
\end{lemma}

The next ingredient is a variance analogue of the mean identity. Its proof is
based on the martingale \(h(X_{t\wedge\tau_y})+t\wedge\tau_y\), and is given in
Section~\ref{sec:4}.

\begin{lemma}\label{lem:variance-edge-identity}
	For every \(x,y\in V\),
	\begin{equation}\label{eq:variance-edge-identity}
		\Var_x(\tau_y)+\E_x\tau_y
		=
		\sum_{\{u,v\}\in E}
		\bigl(g(u)+g(v)\bigr)\bigl(h(u)-h(v)\bigr)^2.
	\end{equation}
\end{lemma}

The final ingredient is an entropy estimate for the unit flow $\theta_g$.

\begin{lemma}\label{lem:entropy-current-bound}
	With \(g\) as above,
	\begin{equation}\label{eq:entropy-current-bound}
		\sum_{\{u,v\}\in E}
		\frac{\bigl(g(u)-g(v)\bigr)^2}{g(u)+g(v)}
		\le
		1+\log n.
	\end{equation}
\end{lemma}

We now prove the theorem from these three ingredients.

\begin{proof}[Proof of Theorem~\ref{thm:general-plus-mean}]
	Write $L=1+\log n$. By Lemma \ref{lem:mean-current-identity},
	\[
	\E_x\tau_y
	=
	\sum_{\{u,v\}\in E}
	\bigl(g(u)-g(v)\bigr)\bigl(h(u)-h(v)\bigr).
	\]
	Applying Cauchy--Schwarz inequality, we obtain
	\[
	(\E_x\tau_y)^2
	\le
	\left[
	\sum_{\{u,v\}\in E}
	\frac{\bigl(g(u)-g(v)\bigr)^2}{g(u)+g(v)}
	\right]
	\left[
	\sum_{\{u,v\}\in E}
	\bigl(g(u)+g(v)\bigr)\bigl(h(u)-h(v)\bigr)^2
	\right].
	\]
	The first factor is bounded by $L$ by
	Lemma~\ref{lem:entropy-current-bound}, while the second factor is equal to
	\(\Var_x(\tau_y)+\E_x\tau_y\) by
	Lemma~\ref{lem:variance-edge-identity}. Therefore
	\[
	(\E_x\tau_y)^2
	\le
	L\bigl(\Var_x(\tau_y)+\E_x\tau_y\bigr),
	\]
	which is equivalent to
	\[
	\Var_x(\tau_y)+\E_x\tau_y
	\ge
	\frac{(\E_x\tau_y)^2}{1+\log n}.\qedhere
	\]
\end{proof}

\begin{proof}[Proof of Remark \ref{rem:remarkmain}]
	In view of \eqref{eq:improve-entropy}, it suffices to repeat the proof of Theorem~\ref{thm:general-plus-mean},
	replacing \(L=1+\log n\) by
	\[
	R:=1+\frac12\log\dist(x,y)+\frac12\log\deg(y)
	=
	\frac{2+\log\dist(x,y)+\log\deg(y)}{2}.\qedhere
	\]
\end{proof}

\section{Proofs of Lemmas \ref{lem:mean-current-identity} and \ref{lem:variance-edge-identity}}\label{sec:4}

Throughout this section, \(h\), \(\mu\), and \(g\) are as in the proof of
Theorem~\ref{thm:general-plus-mean}. Recall that
\[
\theta_g(u,v)=g(u)-g(v)
\]
is a unit flow from \(x\) to \(y\), by \eqref{eq:g-unit-current}.

\begin{proof}[Proof of Lemma~\ref{lem:mean-current-identity}]
	Since \(\theta_g\) is a unit flow from \(x\) to \(y\), the summation-by-parts
	identity \eqref{eq:summation-by-parts} applied to \(f=h\), gives
	\[
	\sum_{\{u,v\}\in E}
	\bigl(g(u)-g(v)\bigr)\bigl(h(u)-h(v)\bigr)
	=
	h(x)-h(y).
	\]
	The result follows by noting that \(h(x)=\E_x\tau_y\) and $h(y)=0$.
\end{proof}

The variance identity comes from the standard martingale associated with the
Poisson equation for \(h\). For every $u \neq y$, the strong Markov property gives 
\begin{equation}\label{eq:poisson}
	h(u)
	=
	1+\frac1{\deg(u)}\sum_{v\sim u}h(v).
\end{equation}
In other words, before the walk hits \(y\), one step of the walk decreases the
expected remaining hitting time by exactly one:
\[
\E_u h(X_1)=h(u)-1,
\quad \forall \, u\neq y.
\]
Thus the decrease of \(h(X_t)\) is compensated by the increase of the time
variable, which gives the following martingale.
\begin{lemma}\label{lem:hitting-time-martingale}
	The process
	\[
	M_t:=h(X_{t\wedge\tau_y})+t\wedge\tau_y
	\]
	is a martingale with respect to the natural filtration $\mathcal F_t:=\sigma(X_0,\ldots,X_t)$.
\end{lemma}

\begin{proof}
	Since \(G\) is finite, \(h\) is bounded, and therefore \(M_t\) is integrable for
	every \(t\). Fix \(t\ge0\). On the event \(\{\tau_y\le t\}\), we have
	\[
	X_{(t+1)\wedge\tau_y}=X_{t\wedge\tau_y}=y,
	\]
	and hence \(M_{t+1}=M_t\).
	
	On the event \(\{\tau_y>t\}\), we have \(X_t\neq y\). Therefore, using
	\eqref{eq:poisson},
	\begin{align*}
		\E_x(M_{t+1}\mid\mathcal F_t)
		&=
		\E_x\bigl(h(X_{t+1})+t+1\mid\mathcal F_t\bigr) \\
		&=
		\frac1{\deg(X_t)}\sum_{u\sim X_t}h(u)+t+1 \\
		&=
		h(X_t)+t
		=
		M_t.
	\end{align*}
	Combining the two cases gives
	\[
	\E_x(M_{t+1}\mid\mathcal F_t)=M_t,
	\]
	so \((M_t)\) is a martingale.
\end{proof}

Since $M_0=\E_x\tau_y$ and $M_{\tau_y}=\tau_y$, the centered hitting time can be written as a sum of martingale increments:
\begin{equation}\label{eq:increments}
\tau_y-\E_x\tau_y
=
\sum_{t=1}^{\tau_y}(M_t-M_{t-1}).
\end{equation}
Orthogonality of martingale increments then gives the following variance
decomposition.

\begin{lemma}\label{lem:martingale-variance-identity}
	For every \(x,y\in V\),
	\begin{equation}\label{eq:martingale-variance-identity}
		\Var_x(\tau_y)
		=
		\E_x\sum_{t=1}^{\tau_y}
		\Var\!\left(h(X_t)\mid\mathcal F_{t-1}\right).
	\end{equation}
\end{lemma}

\begin{proof}
	Recall that Lemma \ref{lem:hitting-time-martingale} shows that $M_t=h(X_{t\wedge\tau_y})+t\wedge\tau_y$ is a martingale. Set
	\[
	\Delta_t:=M_{t}-M_{(t-1)}.
	\]
	Then \((\Delta_t)_{t\ge1}\) are martingale differences and
	\(\Delta_t=0\) on \(\{t>\tau_y\}\). Since the graph is finite,
	\(\tau_y\) has finite second moment, and therefore from (\ref{eq:increments}) we obtain
	\[
	\tau_y-\E_x\tau_y = \sum_{t=1}^{\tau_y} \Delta_t =
	\sum_{t=1}^{\infty}\Delta_t,
	\]
	where the infinite sum is finite almost surely, and the equality holds in \(L^2\). By orthogonality of martingale differences,
	\[
	\Var_x(\tau_y)
	=
	\sum_{t=1}^{\infty}\E_x[\Delta_t^2].
	\]
	On the event \(\{t\le \tau_y\}\), we have
	\[
	\Delta_t
	=
	h(X_t)-h(X_{t-1})+1.
	\]
	Moreover, since \(X_{t-1}\neq y\) on this event, \eqref{eq:poisson}
	gives
	\[
	\E_x\!\left[h(X_t)\mid\mathcal F_{t-1}\right]
	=
	h(X_{t-1})-1.
	\]
	Thus, on \(\{t\le \tau_y\}\),
	\[
	\Delta_t
	=
	h(X_t)-
	\E_x\!\left[h(X_t)\mid\mathcal F_{t-1}\right].
	\]
	Since \(\{t\le\tau_y\}\in\mathcal F_{t-1}\), it follows that
	\[
	\E_x[\Delta_t^2]
	=
	\E_x\!\left[
	\mathbf 1_{\{t\le\tau_y\}}
	\Var\!\left(h(X_t)\mid\mathcal F_{t-1}\right)
	\right].
	\]
	Summing over \(t\ge1\) gives
	\[
	\Var_x(\tau_y)
	=
	\E_x\sum_{t=1}^{\tau_y}
	\Var\!\left(h(X_t)\mid\mathcal F_{t-1}\right),
	\]
	as claimed.
\end{proof}

We now rewrite this martingale variance decomposition as an edge sum using the occupation measure \(\mu\) and the function \(g\).

\begin{proof}[Proof of Lemma~\ref{lem:variance-edge-identity}]
	We rewrite the martingale variance identity as an edge sum. For \(u\neq y\),
	let
	\[
	q(u):=\Var_u(h(X_1)).
	\]
	By Lemma~\ref{lem:martingale-variance-identity} and the Markov property,
	\[
	\Var_x(\tau_y)
	=
	\E_x\sum_{t=1}^{\tau_y}
	\Var\!\left(h(X_t)\mid\mathcal F_{t-1}\right)
	=
	\E_x\sum_{s=0}^{\tau_y-1}q(X_s).
	\]
	Using the occupation measure \(\mu\), this becomes
	\[
	\Var_x(\tau_y)
	=
	\sum_{u\neq y}\mu(u)q(u).
	\]
	Recall that for every $u \neq y$, we have that $\E_u h(X_1)=h(u)-1$. Therefore
	\[
	q(u)
	=
	\frac1{\deg(u)}
	\sum_{v\sim u}
	\bigl(h(v)-h(u)+1\bigr)^2.
	\]
	Since \(\mu(u)=\deg(u)g(u)\), we obtain
	\[
	\Var_x(\tau_y)
	=
	\sum_{u\neq y}\sum_{v\sim u}
	g(u)\bigl(h(v)-h(u)+1\bigr)^2.
	\]
	Expanding the square gives
	\begin{align*}
		\Var_x(\tau_y)
		&=
		\sum_{u \neq y}\sum_{v\sim u}
		g(u)\bigl(h(u)-h(v)\bigr)^2 
		-2\sum_{u \neq y}\sum_{v\sim u}
		g(u)\bigl(h(u)-h(v)\bigr)
		+
		\sum_{u \neq y}\deg(u)g(u).
	\end{align*}
	Note that the third term satisfies
	\[
	\sum_{u \neq y}\deg(u)g(u)
	=
	\sum_{u \neq y}\mu(u)
	=
	\E_x\tau_y.
	\]
	Next, the middle term can be computed using Lemma \ref{lem:mean-current-identity}. Indeed, grouping the two orientations of each edge gives
	\[
	\sum_{u \neq y}\sum_{v\sim u}
	g(u)\bigl(h(u)-h(v)\bigr)
	=
	\sum_{\{u,v\}\in E}
	\bigl(g(u)-g(v)\bigr)\bigl(h(u)-h(v)\bigr)
	=
	\E_x\tau_y.
	\]
	Finally,
	\[
	\sum_{u \neq y}\sum_{v\sim u}
	g(u)\bigl(h(u)-h(v)\bigr)^2
	=
	\sum_{\{u,v\}\in E}
	\bigl(g(u)+g(v)\bigr)
	\bigl(h(u)-h(v)\bigr)^2.
	\]
	Combining these identities, we obtain
	\[
	\Var_x(\tau_y)
	=
	\sum_{\{u,v\}\in E}
	\bigl(g(u)+g(v)\bigr)
	\bigl(h(u)-h(v)\bigr)^2
	-
	\E_x\tau_y,
	\]
	which is equivalent to the desired identity.
\end{proof}

\section{Proof of Lemma \ref{lem:entropy-current-bound}}\label{sec:5}

Recall that \(g(y)=0\) and
\[
\sum_{v\sim u}(g(u)-g(v))=\mathbf 1_{\{u=x\}}-\mathbf 1_{\{u=y\}}.
\]
In particular, \(g\) is the voltage potential of the unit current from \(x\) to
\(y\).

\begin{proof}[Proof of Lemma~\ref{lem:entropy-current-bound}]
Let
\[
C:=\{u\in V:g(u)>0\},
\qquad
N_y:=\{u\in V:u\sim y\}.
\]
Edges with both endpoints outside \(C\) contribute \(0\) to the sum. If \(u\neq y\) is adjacent to a point of \(C\), then \(u\) is reachable from
\(x\) before hitting \(y\), and hence \(g(u)>0\). Thus the only boundary edges
of \(C\) are the edges from \(C\) to \(y\). We first consider the contribution of such edges. Since \(g(y)=0\),
	\begin{equation}\label{eq:boundary}
	\sum_{u \in C\cap N_y}\frac{(g(u)-g(y))^2}{g(u)+g(y)}
	=
	\sum_{u \in C\cap N_y}g(u) = \sum_{u\sim y}(g(u)-g(y))= -\operatorname{div} \theta_g(y)=1.
	\end{equation}
	It remains to bound the contribution of the edges inside \(C\). From the inequality
	\[
	\frac{r-1}{r+1}\le \frac12\log r \quad \forall \, r \geq 1,
	\]
	we infer that
	\begin{equation}\label{eq:ab-ineq}
	\frac{(a-b)^2}{a+b}
	\le
	\frac12(a-b)(\log a-\log b) \quad \forall \, a,b>0.
	\end{equation}
	Since both sides are symmetric in \(a\) and \(b\), we may apply this inequality
	on each unoriented edge. We deduce that
	\begin{equation}\label{eq:entropy-reduction}
		\sum_{\substack{\{u,v\}\in E\\ u,v\in C}}
		\frac{(g(u)-g(v))^2}{g(u)+g(v)}
		\le
		\frac12
		\sum_{\substack{\{u,v\}\in E\\ u,v\in C}}
		(g(u)-g(v))(\log g(u)-\log g(v)).
	\end{equation}
	Consider the function
	\[
	G(z)=
	\begin{cases}
		\log g(z) &\mbox{if } z\in C;\\
		0 &\mbox{if } z\notin C.
	\end{cases}
	\]
	Since \(\theta_g\) is a unit flow from \(x\) to \(y\), the summation-by-parts
	identity \eqref{eq:summation-by-parts} applied to \(G\) gives
	\[
	\sum_{\{u,v\}\in E}
	(g(u)-g(v))(G(u)-G(v))
	=
	G(x)-G(y)
	=
	\log g(x).
	\]
	The only edges with exactly one endpoint in \(C\) are the edges from \(C\) to \(y\).
	Therefore
	\begin{equation}\label{eq: 10.2}
	\sum_{\substack{\{u,v\}\in E\\ u,v\in C}}
	(g(u)-g(v))(\log g(u)-\log g(v)) = 
	\log g(x) - \sum_{u \in C\cap N_y} g(u)\log g(u).
	\end{equation}
	Recall that \(g\) is the unit-current voltage, and thus
	\[
	g(x)=R_{\mathrm{eff}}(x,y)\le \dist(x,y).
	\]
	Moreover, by \eqref{eq:boundary} the vector $(g(u))_{u \in N_y}$ is a probability measure on $N_y$. Hence, using the usual convention $0\log 0 = 0$,
	\[
	-\sum_{u\in C\cap N_y}g(u)\log g(u)
	=
	-\sum_{u\in N_y}g(u)\log g(u)
	\le
	\log |N_y|
	=
	\log\deg(y).
	\]
	Combining this with \eqref{eq:entropy-reduction}, \eqref{eq: 10.2}, and adding the boundary contribution \eqref{eq:boundary}, we conclude that
	\begin{equation}\label{eq:improve-entropy}
	\sum_{\{u,v\}\in E}
	\frac{(g(u)-g(v))^2}{g(u)+g(v)}
	\le
	1+\frac{1}{2} (\log \dist(x,y) + \log \deg(y)) \leq 1+ \log n.\qedhere
	\end{equation}
\end{proof}

\section{Proof of Corollary \ref{thm:bounded-degree}}\label{sec:6}

The following estimate allows us to remove the additive mean term
when the degree is bounded.

\begin{lemma}\label{lem:linear}
	Let $G$ be a graph with maximum degree at most \(D\). If \(\tau_y\) is not deterministic, then
	\[
	\Var_x(\tau_y)
	\ge
	\frac{\E_x\tau_y}{3D}.
	\]
\end{lemma}

\begin{proof}
	Recall that $\E_u h(X_1)=h(u)-1$ for \(u\neq y\). Thus,
	\[
	q(u):=\Var_u(h(X_1))
	=
	\frac1{\deg(u)}
	\sum_{v\sim u}
	\bigl(h(u)-h(v)-1\bigr)^2.
	\]
	Let \(u,v\neq y\) be adjacent vertices, and set $a=h(u)-h(v)$. The edge \(\{u,v\}\) contributes at least
	\[
	\frac{(a-1)^2}{\deg(u)} \quad \mbox{and} \quad \frac{(a+1)^2}{\deg(v)}
	\]
	to \(q(u)\) and $q(v)$, respectively. Applying this to the edge \(\{X_t,X_{t+1}\}\), with \(a=h(X_t)-h(X_{t+1})\), shows that for any $t \leq \tau_y-2$ we have
	\[ q(X_t) + q(X_{t+1}) \geq \frac{(a-1)^2}{D} + \frac{(a+1)^2}{D} \geq \frac{2}{D}.\]
	In view of Lemma \ref{lem:martingale-variance-identity}, we get
	\[
	\Var_x(\tau_y)=\E_x\sum_{t<\tau_y}q(X_t) \geq \frac{1}{2} \E_x \sum_{t=0}^{\tau_y-2}
	\bigl(q(X_t)+q(X_{t+1})\bigr)
	\ge \E_x
	\frac{(\tau_y-1)}{D} = \frac{\E_x \tau_y - 1}{D}.
	\]
	It remains to remove the \(-1\). If \(\tau_y\) is not deterministic, then \(x\) has some neighbor different from $y$. Therefore,
	\[
	\Pp_x(\tau_y\geq 2)\ge \frac{1}{2},
	\]
	which clearly implies that $\E_x \tau_y \geq 3/2$. Combining this with the previous bound gives
	\[
	\Var_x(\tau_y)\ge \frac{\E_x\tau_y - 1}{D} \geq \frac{\E_x \tau_y}{3D}.\qedhere
	\]
\end{proof}

\begin{proof}[Proof of Corollary~\ref{thm:bounded-degree}]
	For convenience, write
	\[
	m:=\E_x\tau_y,
	\qquad
	V:=\Var_x(\tau_y),
	\qquad
	R:=2+\log \dist(x,y)+\log D.
	\]
	Applying Lemma~\ref{lem:linear} and Remark~\ref{rem:remarkmain}, we obtain
	\[
	(1+3D)V
	=
	V+3DV
	\ge
	V+m
	\ge
	\frac{2m^2}{R}.
	\]
	Since \(D\) is fixed, the term \(\log D\) can be absorbed into the constant. More precisely, it is easy to see that
	\[
	R
	\le
	(2+\log D)\bigl(1+\log(\dist(x,y)+1)\bigr) \leq 3(2+\log D) \log(\dist(x,y) +1),
	\]
	where the last inequality uses \(\dist(x,y)\ge1\). Therefore
	\[
	V
	\ge
	\frac{2}{3(1+3D)(2+\log D)}
	\frac{m^2}{\log(\dist(x,y)+1)}\geq \frac{1}{6D(2+\log D)} \frac{m^2}{\log(\dist(x,y)+1)}.\qedhere
	\]
\end{proof}

\section{A high-degree obstruction}\label{sec:funnel}

In this section we construct a family of high-degree graphs for which
\(\mathbb E_x\tau_y\) grows linearly while \(\operatorname{Var}_x(\tau_y)\)
remains bounded. This shows that
Corollary~\ref{thm:bounded-degree} cannot hold with a degree-independent constant.

Fix integers \(L\ge2\) and \(B\ge2\). Let \(G_{L,B}\) be the graph obtained from
the rooted \(B\)-ary tree of depth \(L\), with root \(x\), by adding a new vertex
\(y\) and connecting every depth-\(L\) leaf to \(y\). See Figure~\ref{fig:funnel}
for the case $L=3$ and \(B=2\).

\begin{figure}[b]
	\centering
	\begin{tikzpicture}[
		scale=0.62,
		vertex/.style={circle, draw, fill=white, inner sep=1.4pt, minimum size=5.5pt},
		target/.style={circle, draw, fill=white, inner sep=2pt, minimum size=7pt},
		levelnode/.style={circle, draw, fill=white, inner sep=2pt, minimum size=16pt},
		every node/.style={font=\small}
		]
		
		
		\node at (3.8,3.55) {\textbf{\(G_{3,2}\)}};
		
		\node[vertex,label=left:$x$] (r) at (0,0) {};
		
		\node[vertex] (a) at (1.7,1.35) {};
		\node[vertex] (b) at (1.7,-1.35) {};
		
		\node[vertex] (aa) at (3.4,2.05) {};
		\node[vertex] (ab) at (3.4,0.65) {};
		\node[vertex] (ba) at (3.4,-0.65) {};
		\node[vertex] (bb) at (3.4,-2.05) {};
		
		\node[vertex] (aaa) at (5.1,2.45) {};
		\node[vertex] (aab) at (5.1,1.75) {};
		\node[vertex] (aba) at (5.1,0.95) {};
		\node[vertex] (abb) at (5.1,0.25) {};
		\node[vertex] (baa) at (5.1,-0.25) {};
		\node[vertex] (bab) at (5.1,-0.95) {};
		\node[vertex] (bba) at (5.1,-1.75) {};
		\node[vertex] (bbb) at (5.1,-2.45) {};
		
		\node[target,label=right:$y$] (y) at (7.0,0) {};
		
		\draw (r) -- (a);
		\draw (r) -- (b);
		
		\draw (a) -- (aa);
		\draw (a) -- (ab);
		\draw (b) -- (ba);
		\draw (b) -- (bb);
		
		\draw (aa) -- (aaa);
		\draw (aa) -- (aab);
		\draw (ab) -- (aba);
		\draw (ab) -- (abb);
		\draw (ba) -- (baa);
		\draw (ba) -- (bab);
		\draw (bb) -- (bba);
		\draw (bb) -- (bbb);
		
		\draw (aaa) -- (y);
		\draw (aab) -- (y);
		\draw (aba) -- (y);
		\draw (abb) -- (y);
		\draw (baa) -- (y);
		\draw (bab) -- (y);
		\draw (bba) -- (y);
		\draw (bbb) -- (y);

		
		\node at (11.8,3.55) {\textbf{Level process}};
		
		\node[levelnode] (l0) at (9.0,0) {\(0\)};
		\node[levelnode] (l1) at (10.4,0) {\(1\)};
		\node[levelnode] (l2) at (11.8,0) {\(2\)};
		\node[levelnode] (l3) at (13.2,0) {\(3\)};
		\node[levelnode] (ly) at (14.6,0) {\(y\)};
		
		\draw (l0) -- (l1);
		\node at (9.7,0.25) {\scriptsize \(1\)};
		
		\draw (l1) to[bend left=18] (l2);
		\draw (l1) to[bend right=18] (l2);
		\node at (11.1,0.45) {\scriptsize \(2\)};
		
		\draw (l2) to[bend left=32] (l3);
		\draw (l2) to[bend left=12] (l3);
		\draw (l2) to[bend right=12] (l3);
		\draw (l2) to[bend right=32] (l3);
		\node at (12.5,0.65) {\scriptsize \(4\)};
		
		\draw (l3) to[bend left=32] (ly);
		\draw (l3) to[bend left=12] (ly);
		\draw (l3) to[bend right=12] (ly);
		\draw (l3) to[bend right=32] (ly);
		\node at (13.9,0.65) {\scriptsize \(4\)};

	\end{tikzpicture}
	\caption{The graph \(G_{3,2}\), together with a weighted multiedge representation
		of its level process. The edge multiplicities \(1,2,4,4\) give transition
		probabilities after normalization at each level.}
	\label{fig:funnel}
\end{figure}
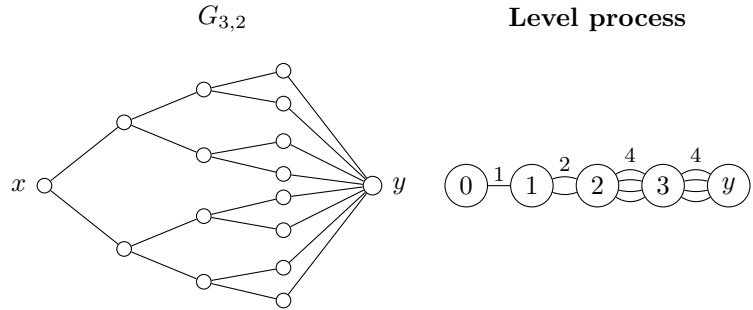

The level process is a birth-death chain on
\[
\{0,1,\ldots,L,L+1\},
\]
where \(L+1\) corresponds to \(y\). From \(0\), the walk moves to \(1\)
deterministically. For \(1\le i\le L-1\), it moves to \(i+1\) with probability
\[
p:=\frac{B}{B+1}
\]
and to \(i-1\) with probability
\[
q:=\frac{1}{B+1}.
\]
From \(L\), it moves to \(L+1\) with probability \(1/2\), and to \(L-1\) with
probability \(1/2\).

Let \(T_0\) be the hitting time of level \(L\), starting from level \(0\). Similarly, let \(T_{L-1}\) be the hitting time of level \(L\), starting from level \(L-1\).

\begin{lemma}\label{lem:funnel-crossings}
	For $L\geq 2$ and $B\geq2$, consider the graph $G_{L,B}$ as above. Then
	\[
	\E T_0\le 3L,
	\qquad
	\Var(T_0)\le \frac{48L}{B},
	\]
	and
	\[
	\E T_{L-1}\le 3,
	\qquad
	\Var(T_{L-1})\le \frac{48}{B}.
	\]
\end{lemma}

\begin{proof}
	For \(0\le i\le L-1\), let \(A_i\) be the hitting time of level \(i+1\) for the level chain started from level \(i\). We claim that
	\begin{equation}\label{eq:A_i}
		\E A_i \leq 3 \quad \mbox{and} \quad \Var A_i\leq \frac{48}{B} \quad \forall \, i=0,\ldots,L-1.
	\end{equation} 
	Clearly \(A_0=1\).  For \(i\ge1\), starting from level \(i\), each attempt to move
	to level \(i+1\) succeeds with probability $p=B/(B+1)$ and fails with probability $q=1/(B+1)$. On a failed attempt, the walk first moves to \(i-1\), then must return from
	\(i-1\) to \(i\), which takes an independent copy of \(A_{i-1}\), and then it
	tries again. Therefore
	\[
	A_i
	\stackrel{d}{=}
	1+\sum_{r=1}^{G_i}\bigl(1+A_{i-1}^{(r)}\bigr),
	\]
	where \(G_i\) is geometric with success parameter \(p\). From the representation above,
	\[
	\E A_i
	=1+\E(G_i)(1+\E A_{i-1})=
	1+\frac1B(1+\E A_{i-1}).
	\]
	Since \(\E A_0=1\) and \(B\ge2\), induction gives \(\E A_i\le3\) for all \(0\leq i \leq L-1\). For the variance, using the compound-sum formula and the previous estimate,
	\begin{align*}
	\Var A_i
	&=
	\E G_i\,\Var(1+A_{i-1})
	+
	\Var(G_i)\bigl(\E(1+A_{i-1})\bigr)^2 = \frac1B \Var(A_{i-1})
	+
	\frac{B+1}{B^2}(1+\E A_{i-1})^2 \\
	&\leq \frac{1}{B} (\Var(A_{i-1}) + 24)
	\end{align*}
	Since \(\Var A_0=0\) and $B\geq 2$, induction gives $\Var A_i \leq 48/B$ for all $0\leq i\leq L-1$, and the claim follows. Note that \(T_{L-1}=A_{L-1}\), and thus
	\[
	\E T_{L-1}\le 3,
	\qquad
	\Var(T_{L-1})\le \frac{48}{B}.
	\]
	Finally, by the strong Markov property,
	\[
	T_0 \stackrel{d}{=} A_0+A_1+\cdots+A_{L-1},
	\]
	with independent summands. Therefore, from (\ref{eq:A_i}) we get
	\[
	\E T_0=\sum_{i=0}^{L-1}\E A_i\le 3L \qquad \mbox{and} \qquad \Var(T_0)=\sum_{i=0}^{L-1}\Var(A_i)\le \frac{48L}{B}.\qedhere
	\]
\end{proof}

\begin{proof}[Proof of Proposition~\ref{prop:funnel}]
	For each \(L\ge2\), set \(B=L^4\). Let \(G_L=G_{L,B}\) be the graph defined as above. Recall that \(T_0\) denotes the hitting time of level \(L\)
	starting from level \(0\), and \(T_{L-1}\) is the hitting time of level \(L\)
	starting from level \(L-1\).
	
	After the walk first reaches level \(L\), each attempt to jump to \(y\) succeeds
	with probability \(1/2\). If the attempt fails, the walk moves to level \(L-1\)
	and must return to level \(L\). Hence, by the strong Markov property,
	\[
	\tau_y
	\stackrel{d}{=}
	T_0+1+\sum_{r=1}^{N}\bigl(1+T_{L-1}^{(r)}\bigr),
	\]
	where \(N\) is geometric with parameter \(1/2\), counting failures before the
	first success, and the variables on the right-hand side are independent. Taking expectation and using Lemma~\ref{lem:funnel-crossings},
	\[
	\E_x \tau_y = \E T_0 + 1 + \E(N)(1+\E(T_{L-1})) \leq 3L + 5.
	\]
	Since also \(\tau_y\ge L+1\), it follows that
	\[
	\E_x\tau_y\asymp L.
	\]
	For the variance, note that $T_0$, $N$, and $T_{L-1}^{(r)}$ are independent. Thus, using the compound-sum formula and Lemma \ref{lem:funnel-crossings} gives
	\begin{align*}
	\Var_x(\tau_y) &= \Var(T_0) +
	\Var\left(\sum_{r=1}^{N} (1+ T_{L-1}^{(r)})\right)
	= \Var(T_0) + \E N \Var(T_{L-1}) + \Var(N) (1+\E T_{L-1})^2\\
	&= O\left(\frac{L}{B}\right)+O\left(\frac1B\right)+O(1)
	=
	O(1).
	\end{align*}
	It remains to estimate the size of \(G_L\). Since \(G_L\) consists of the rooted
	\(B\)-ary tree of depth \(L\), together with the additional vertex \(y\),
	\[
	|V_L|
	=
	1+\sum_{r=0}^{L}B^r
	\asymp B^L.
	\]
	Because \(B=L^4\), we get
	\[
	\log |V_L|
	\asymp
	L\log B
	\asymp
	L\log L.\qedhere
	\]
\end{proof}

\begin{remark}\label{rem:local-counterexample}
	Norris--Peres--Zhai conjectured \cite[Conjecture~6.1]{norris2017surprise} that the general
	logarithmic bound for surprise probabilities could be improved to a square-root
	logarithmic bound, namely that
	\begin{equation}\label{eq:conjecture}
	\Pp_x(\tau_y=t)
	\le
	C\frac{\sqrt{1+\log |V|}}{t}
	\end{equation}
	uniformly over finite graphs and all choices of \(x,y,t\). Such an estimate
	would imply, by a standard Chebyshev argument, a variance lower bound of the form
	\[
	\Var_x(\tau_y)
	\ge
	c\frac{(\E_x\tau_y)^2}{1+\log |V|}.
	\]
	Proposition~\ref{prop:funnel} shows that this variance consequence is false
	without a bounded-degree assumption.
	
	In fact, the same example gives a direct counterexample to the local estimate.
	For the graphs \(G_L=G_{L,L^4}\), the shortest possible hitting time from $x$ to $y$ is \(L+1\).
	This occurs exactly when the walk moves away from the root at every internal
	level and then jumps from level \(L\) to \(y\). Therefore
	\[
	\Pp_x(\tau_y=L+1)
	=
	\frac12\left(\frac{B}{B+1}\right)^{L-1}.
	\]
	With \(B=L^4\), this probability tends to \(1/2\). Hence
	\[
	\sup_t t\,\Pp_x(\tau_y=t)
	\ge
	(L+1)\Pp_x(\tau_y=L+1)
	\ge
	(1-o(1))\frac L2.
	\]
	On the other hand, $\sqrt{\log |V_L|}\asymp \sqrt{L\log L}$, and thus
	\[
	\frac{\sup_t t\,\Pp_x(\tau_y=t)}
	{\sqrt{1+\log |V_L|}}
	\to\infty.
	\]
	Consequently, the square-root logarithmic local estimate cannot hold over all
	finite graphs without an additional hypothesis, such as bounded degree. Whether (\ref{eq:conjecture}) holds under a bounded degree assumption remains open.
\end{remark}

\section{Proof of Corollary \ref{theo:tree}}\label{sec:8}

We will use the same martingale and
flow identities as above, but on a tree the current has a particularly simple
form. Let $x\neq y$ be vertices of a tree $T$ and let
\[
x=u_0,u_1,\ldots,u_\ell=y
\]
be the geodesic from \(x\) to \(y\), so that \(\ell=\dist(x,y)\).  Since \(T\) is
a tree, the unit flow $\theta_g$ flows only along this geodesic.
With the normalization \(g(y)=0\), it is standard that
\[
g(u_j)=\ell-j,
\qquad j=0,\ldots,\ell.
\]
Moreover, \(g\) is constant on every component attached to the geodesic.

We first show that the variance dominates the expectation.

\begin{lemma}\label{lem:tree-linear}
	Let $x\neq y$ be vertices of a tree $T$. If \(\tau_y\) is not deterministic, then
	\[
	\Var_x(\tau_y)\ge 2\E_x\tau_y.
	\]
\end{lemma}

\begin{proof}
	Recall that $h(u)=\E_u(\tau_y)$ and write $\delta_j = h(u_{j})-h(u_{j+1})$ for $j=0,\ldots,\ell-1$. Decomposing $\tau_y$ as a sum of hitting times across the geodesic, we see that
	\[
	\E_x\tau_y= \sum_{j=0}^{\ell-1} \E_{u_{j}} (\tau_{u_{j+1}}) = \sum_{j=0}^{\ell-1} \delta_j
	\]
	Using Lemma~\ref{lem:variance-edge-identity} and keeping only the geodesic
	edges, we get
	\[
	\Var_x(\tau_y)+\E_x\tau_y
	\ge
	\sum_{j=0}^{\ell-1}
	\bigl(g(u_{j})+g(u_{j+1})\bigr)\delta_j^2.
	\]
	Since $T$ is a tree,
	\[
	g(u_{j})+g(u_{j+1})
	=
	(\ell-j)+(\ell-j-1)
	=
	2(\ell-j)-1.
	\]
	We obtain
	\[
	\Var_x(\tau_y) \geq \sum_{j=0}^{\ell-1} \left ((2(\ell-j)-1)\delta_j^2 - \delta_j\right ).
	\]
	For \(j\leq\ell-2\), we have \(2(\ell-j)-1\ge3\). Since \(\delta_j\ge1\), it follows that
	\[
	\bigl(2(\ell-j)-1\bigr)\delta_j^2-\delta_j
	\ge
	3\delta_j^2-\delta_j
	\ge
	2\delta_j.
	\]
	For the final edge \(j=\ell-1\), if \(\tau_y\) is not deterministic, then
	\(u_{\ell-1}\) has at least one neighbor different from \(y\).  Starting from
	\(u_{\ell-1}\), the walk makes a geometric number of failed attempts before
	stepping to \(y\), and each failed attempt costs at least two steps. Hence
	\[
	\delta_{\ell-1}=\E_{u_{\ell-1}}\tau_y\ge 3.
	\]
	Therefore
	\[
	\delta_{\ell-1}^2-\delta_{\ell-1}\ge 2\delta_{\ell-1}.
	\]
	Combining the estimates for all \(j\), we get
	\[
	\Var_x(\tau_y)
	\ge
	2\sum_{j=0}^{\ell-1}\delta_j
	=
	2\E_x\tau_y.\qedhere
	\]
\end{proof}

\begin{proof}[Proof of Corollary~\ref{theo:tree}]
	Let $x=u_0, u_1,\ldots, u_\ell=y$ be the geodesic from \(x\) to \(y\), so that \(\ell=\dist(x,y)\).
	Looking at the proof of Lemma~\ref{lem:entropy-current-bound}, equations
	\eqref{eq:entropy-reduction} and \eqref{eq: 10.2} give
	\[
	\sum_{\substack{\{u,v\}\in E\\ u,v\in C}}
	\frac{(g(u)-g(v))^2}{g(u)+g(v)}
	\le
	\frac12\left(
	\log g(x)
	-
	\sum_{\substack{u\sim y\\ u\in C}}g(u)\log g(u)
	\right).
	\]
	Since \(T\) is a tree, the unit flow $\theta_g$ flows only along
	the geodesic. Hence
	\[
	g(x)=R_{\mathrm{eff}}(x,y)=\dist(x,y)=\ell.
	\]
	Moreover, the only nonzero coordinate of \((g(u))_{u\sim y}\) is
	\(g(u_{\ell-1})=1\). Therefore
	\[
	\sum_{\substack{u\sim y\\ u\in C}}g(u)\log g(u)=0.
	\]
	Adding the boundary contribution from \eqref{eq:boundary} and using that $\ell\geq1$, we get
	\[
	\sum_{\{u,v\}\in E}
	\frac{(g(u)-g(v))^2}{g(u)+g(v)}
	\le
	1+\frac12\log \ell
	\le
	\frac{\log(\ell+1)}{\log 2}.
	\]
	Write
	\[
	m:=\E_x\tau_y,
	\qquad
	V:=\Var_x(\tau_y),
	\qquad
	L:= \frac{\log(\ell+1)}{\log 2}.
	\]
	Recall that Lemma~\ref{lem:tree-linear} gives $V\geq 2m$. Hence, by the Cauchy--Schwarz argument used in Theorem~\ref{thm:general-plus-mean},
	\[
	\frac32 V\ge V+m\ge \frac{m^2}{L}.
	\]
	Therefore
	\[
	V\ge \frac23\frac{m^2}{L}.\qedhere
	\]
\end{proof}

\appendix

\section{A generalization for reversible Markov chains}\label{app:reversible}

In this section, we record a weighted-network version of the argument behind Theorem \ref{thm:general-plus-mean}. This is the natural form of the proof for finite reversible Markov chains.

Let \(X_t\) be an irreducible reversible Markov chain taking values on a finite state space
\(\Omega\) with stationary measure \(\pi\) and transition matrix $P$. For \(u\neq v \in \Omega\), define the
conductance
\[
c(u,v):=\pi(u)P(u,v)=\pi(v)P(v,u).
\]
Fix a starting point \(x \in \Omega\) and a target point $y \in \Omega$ with $x\neq y$. For $u \in \Omega$, define
\[
h(u)=\E_u \tau_y,
\qquad
\mu(u):=\E_x\#\{0\le t<\tau_y:X_t=u\},
\qquad
g(u):=\frac{\mu(u)}{\pi(u)}.
\]
In particular, $h(y)=0$ and $g(y)=0$. For an oriented edge \((u,v)\), the expected number of crossings from \(u\) to
\(v\) before time \(\tau_y\) is
\[
F(u,v)=\mu(u)P(u,v)=c(u,v)g(u).
\]
For any stopped path from \(x\) to \(y\), the net number of exits from a vertex
\(u\) minus the number of entrances into \(u\) is $\mathbf 1_{\{u=x\}}-\mathbf 1_{\{u=y\}}$. Taking expectations gives
\[
\mathbf 1_{\{u=x\}}-\mathbf 1_{\{u=y\}} = \sum_{v \neq u} \bigl(F(u,v)-F(v,u)\bigr)
= \sum_{v\neq u} c(u,v)(g(u)-g(v))
\]
In other words, 
\[
\theta(u,v):=
c(u,v)(g(u)-g(v))
\]
is a unit flow from \(x\) to \(y\). Thus \(g\) is the unit-current voltage in the conductance network \(c\), normalized by \(g(y)=0\). In particular,
\[
g(x)=R_{\mathrm{eff}}(x,y),
\]
where \(R_{\mathrm{eff}}(x,y)\) denotes effective resistance in this weighted
network. First, we obtain weighted versions of Lemmas \ref{lem:mean-current-identity} and \ref{lem:variance-edge-identity}.

\begin{lemma}\label{lem:rev-mean-var}
	Let \(P\) be a finite irreducible reversible Markov chain with state space $\Omega$, and let $x$, $y \in \Omega$. Then,
	\begin{equation}\label{eq:rev-mean}
		\E_x\tau_y
		=
		\sum_{\{u,v\}}
		c(u,v)(g(u)-g(v))(h(u)-h(v)),
	\end{equation}
	and
	\begin{equation}\label{eq:rev-var}
		\Var_x(\tau_y)+\E_x\tau_y
		=
		\sum_{\{u,v\}}
		c(u,v)(g(u)+g(v))(h(u)-h(v))^2.
	\end{equation}
	Here the sums are over unordered pairs \(\{u,v\}\) with \(u\neq v\) and
	\(c(u,v)>0\).
\end{lemma}

\begin{proof}
	We showed above that \(\theta\) is a unit flow from \(x\) to \(y\). The summation-by-parts identity (\ref{eq:summation-by-parts}) applied to the function $h$ gives
	\[
	\sum_{\{u,v\}}
	c(u,v)(g(u)-g(v))(h(u)-h(v))
	=
	h(x)-h(y).
	\]
	Since \(h(x)=\E_x\tau_y\) and \(h(y)=0\), this proves
	\eqref{eq:rev-mean}.
	
	It remains to prove the variance identity.  The martingale variance identity
	from Lemma~\ref{lem:martingale-variance-identity} applies verbatim to any finite
	Markov chain, and gives
	\[
	\Var_x(\tau_y)
	=
	\sum_{u\neq y}\mu(u)\Var_u(h(X_1)).
	\]
	Recall that for any $u\neq y$ we have $\E_u h(X_1)=h(u)-1$, and thus
	\[
	\Var_u(h(X_1))
	=
	\sum_{v \in \Omega} P(u,v)\bigl(h(u)-h(v)-1\bigr)^2.
	\]
	Adding over $u\neq y$ and expanding the square gives
	\begin{align*}
		\Var_x(\tau_y)
		&=
		\sum_{u\neq y}\sum_{v \in \Omega}
		\mu(u)P(u,v)(h(u)-h(v))^2
		-2\sum_{u\neq y}\sum_{v \in \Omega}
		\mu(u)P(u,v)(h(u)-h(v))
		+
		\sum_{u\neq y}\mu(u).
	\end{align*}
	The last term is \(\E_x\tau_y\). The middle term is also \(\E_x\tau_y\) because for any $u\neq y$,
	\[
	\sum_{v \in \Omega}P(u,v)(h(u)-h(v))=1.
	\]
	For the first term, self-loops contribute nothing, and for \(u\neq v\),
	\[
	\mu(u)P(u,v)=g(u)\pi(u)P(u,v)=g(u)c(u,v).
	\]
	Grouping the two orientations of each edge gives
	\[
	\sum_{u\neq y}\sum_v
	\mu(u)P(u,v)(h(u)-h(v))^2
	=
	\sum_{\{u,v\}}
	c(u,v)(g(u)+g(v))(h(u)-h(v))^2.
	\]
	Combining the three terms proves \eqref{eq:rev-var}.
\end{proof}

\begin{theorem}\label{theo:reversible}
	Let \(P\) be a finite irreducible reversible Markov chain with state space $\Omega$, and let $x$, $y\in \Omega$ with $x \neq y$. Set
	\[
	C_y:=\sum_{u\neq y}c(u,y)=\pi(y)(1-P(y,y)).
	\]
	Then
	\[
	\Var_x(\tau_y)+\E_x\tau_y
	\ge
	\frac{(\E_x\tau_y)^2}
	{1+\frac12\log\!\left(R_{\mathrm{eff}}(x,y)C_y\right)}.
	\]
\end{theorem}

\begin{proof}
	The proof is the same as in the unweighted case, with each edge weighted by its
	conductance. Applying Cauchy--Schwarz to \eqref{eq:rev-mean} we get
	\begin{equation}\label{eq:c-s}
	(\E_x\tau_y)^2
	\le
	L
	\bigl(\Var_x(\tau_y)+\E_x\tau_y\bigr),
	\end{equation}
	where
	\[
	L:=
	\sum_{\{u,v\}}
	c(u,v)\frac{(g(u)-g(v))^2}{g(u)+g(v)},
	\]
	where terms with \(g(u)=g(v)=0\) are understood to contribute \(0\). Thus it remains to bound \(L\). 
	
	Let $C:=\{u\in\Omega \colon g(u)>0\}$ and let $N_y=\{u \in \Omega \colon u\neq y, c(u,y)>0\}$ be the set of neighbors of $y$ different from $y$. If \(u\neq y\) is adjacent to a point
	of \(C\), then \(u\) is reachable from \(x\) before hitting \(y\), and hence
	\(g(u)>0\). Thus the only boundary edges of \(C\) are the edges from \(C\) to
	\(y\). The boundary contribution at \(y\) is
	\[
	\sum_{u \in C \cap N_y}
	c(u,y)\frac{(g(u)-g(y))^2}{g(u)+g(y)}
	=
	\sum_{u \in C \cap N_y} c(u,y)g(u)
	=
	1,
	\]
	because the total current entering \(y\) is one. For edges inside \(C\), applying the inequality (\ref{eq:ab-ineq}) with $a=g(u)$ and $b=g(v)$ we obtain
	\[
	\sum_{\substack{\{u,v\}\\u,v\in C}}
	c(u,v)\frac{(g(u)-g(v))^2}{g(u)+g(v)}
	\le
	\frac12
	\sum_{\substack{\{u,v\}\\u,v\in C}}
	c(u,v)(g(u)-g(v))(\log g(u)-\log g(v)).
	\]
	Consider the function
	\[
	G(z)=
	\begin{cases}
		\log g(z) & \mbox{if } z\in C;\\
		0 & \mbox{if } z\notin C.
	\end{cases}
	\]
	Since $\theta(u,v)=c(u,v)(g(u)-g(v))$ is a unit flow from \(x\) to \(y\), the summation-by-parts identity (\ref{eq:summation-by-parts}) applied to $G$ yields
	\[
	\sum_{\{u,v\}}
	c(u,v)(g(u)-g(v))(G(u)-G(v))
	=
	G(x)-G(y)
	=
	\log g(x).
	\]
	The only boundary edges of \(C\) are the edges from \(C\) to \(y\), and therefore
	\[
	\sum_{\substack{\{u,v\}\\u,v\in C}}
	c(u,v)(g(u)-g(v))(\log g(u)-\log g(v))
	=
	\log g(x)-\sum_{u \in C \cap N_y}c(u,y)g(u)\log g(u).
	\]
	After combining this estimate with the boundary contribution at $y$, we obtain
	\[L
	\le
	1+\frac12\left(
	\log g(x)-\sum_{u \in C \cap N_y}c(u,y)g(u)\log g(u)
	\right).
	\]
	On the one hand, since $g$ is the unit voltage and $g(y)=0$, we deduce that $g(x)=R_{\mathrm{eff}}(x,y)$. On the other hand, using the usual convention $0 \log 0 = 0$, Gibbs' inequality applied to the probability vectors
	\[
	\nu(u)=c(u,y)g(u),
	\qquad
	\rho(u)=\frac{c(u,y)}{C_y}
	\qquad \forall \, u \in N_y,
	\]
	shows that
	\[
	0\le \sum_{u \in N_y}\nu(u)\log\frac{\nu(u)}{\rho(u)} =\sum_{u \in C \cap N_y}\nu(u)\log\frac{\nu(u)}{\rho(u)}
	=
	\sum_{u \in C \cap N_y}\nu(u)\log\nu(u)
	-
	\sum_{u \in C \cap N_y}\nu(u)\log c(u,y)
	+
	\log C_y.
	\]
	Thus
	\[
	-\sum_{u \in C \cap N_y}\nu(u)\log g(u)\le \log C_y.
	\]
	We conclude that
	\[
	L
	\le
	1+\frac12\log\!\left(R_{\mathrm{eff}}(x,y)C_y\right).
	\]
	Combining this with the Cauchy--Schwarz estimate (\ref{eq:c-s}) proves the theorem.
\end{proof}

\begin{remark}
	The quantity \(R_{\mathrm{eff}}(x,y)C_y\) measures how much larger the full resistance from $x$ to $y$ is than the unavoidable final resistance of entering $y$. It also has a probabilistic interpretation (see \cite[Proposition 9.5]{levin2017markov}): If $\tau_y^+:=\inf\{t\ge1:X_t=y\}$, then
	\[
	\frac{1}{R_{\mathrm{eff}}(x,y)C_y}
	=
	\mathbb P_y(\tau_x<\tau_y^+ \mid X_1\neq y).
	\]
	Thus \(R_{\mathrm{eff}}(x,y)C_y\) is the inverse probability that an excursion
	from \(y\), conditioned to leave \(y\), reaches \(x\) before returning to \(y\).
	In particular \(R_{\mathrm{eff}}(x,y)C_y\ge1\).
\end{remark}

\section{Failure of the interval conjecture}\label{app:interval-conjecture}

Norris--Peres--Zhai record the following interval conjecture of Holroyd
\cite[Conjecture~6.2]{norris2017surprise}. It asks whether, for a Markov chain with \(n\) states,
\begin{equation}\label{eq:interval-conjecture}
	\mathbb P_x(t\le \tau_y\le t+n)
	\le
	C\frac{n}{t},
	\qquad t>n,
\end{equation}
with a universal constant \(C\). We show that this fails even for simple random
walk on bounded-degree trees.

\begin{proposition}\label{prop:interval-counterexample}
	There exist bounded-degree trees \(T_m\), vertices \(x_m,y_m\in T_m\), integers
	\(t_m>|V(T_m)|\), and a universal constant $c>0$, such that, writing \(n_m:=|V(T_m)|\),
	\[
	\mathbb P_{x_m}\left(t_m\le \tau_{y_m}\le t_m+n_m\right)
	\ge
	\frac{c}{\sqrt{\log n_m}},
	\qquad 
	\frac{n_m}{t_m}\asymp \frac{1}{\log n_m}.
	\]
	Consequently, \eqref{eq:interval-conjecture} cannot hold uniformly over finite
	Markov chains, even restricted to simple random walk on bounded-degree trees.
\end{proposition}

\begin{proof}
	We use the tree \(G_m\) from the proof of
	\cite[Claim~2.2]{norris2017surprise}, whose maximal degree is $4$. Set $T_m=G_m$. With the notation of that paper, let
	\(x_m=w_m\), \(y_m=w_0\), and \(n_m=|V(G_m)|\).  By
	\cite[Lemma~2.3]{norris2017surprise},
	\[
	\mathbb E_{w_m}\tau_{y_m}=\Theta(m2^{2m}),
	\qquad
	\operatorname{Var}_{w_m}(\tau_{y_m})=O(m2^{4m}).
	\]
	Moreover, the construction has
	\[
	n_m=2^{2m}-2^m-2m+2\asymp 2^{2m}.
	\]
	Equivalently,
	\begin{equation}\label{eq:peres}
	\mathbb E_{x_m}\tau_{y_m}\asymp n_m\log n_m,
	\qquad
	\sqrt{\operatorname{Var}_{x_m}(\tau_{y_m})}
	\lesssim n_m\sqrt{\log n_m}.
	\end{equation}
	By Chebyshev's inequality, there is a constant \(a>0\) such that
	\[
	\mathbb P_{x_m}\left(
	\left|\tau_{y_m}-\mathbb E_{x_m}\tau_{y_m}\right|
	\le
	a n_m\sqrt{\log n_m}
	\right)
	\ge
	\frac12.
	\]
	Thus a fixed positive amount of the mass of \(\tau_{y_m}\) lies in an interval of
	length \(O(n_m\sqrt{\log n_m})\).  Partition this interval into subintervals of
	length \(n_m\).  Since the number of such subintervals is
	\(O(\sqrt{\log n_m})\), there exists a time \(t_m\) and a constant $c>0$ such that
	\[
	\mathbb P_{x_m}\left(t_m\le \tau_{y_{y_m}}\le t_m+n_m\right)
	\ge
	\frac{c}{\sqrt{\log n_m}}.
	\]
	Moreover, the interval above is centered at $\mathbb E_{x_m}\tau_{y_m}\asymp n_m\log n_m$
	and has width only \(O(n_m\sqrt{\log n_m})\). Hence, for large \(m\),
	\[
	t_m\asymp n_m\log n_m.
	\]
	In particular \(t_m>n_m\), and
	\[
	\frac{n_m}{t_m}\asymp \frac{1}{\log n_m}.
	\]
	Therefore
	\[
	\frac{ \mathbb P_{x_m}\left(t_m\le \tau_{y_{y_m}}\le t_m+n_m\right)
	}{
		n_m/t_m
	}
	\gtrsim
	\sqrt{\log n_m}
	\to\infty.
	\]
	Thus, no universal constant $C>0$ can make \eqref{eq:interval-conjecture} hold.
\end{proof}

\bibliographystyle{plain}	
\bibliography{thesisbib}

@book {levin2017markov,
	AUTHOR = {Levin, David A. and Peres, Yuval},
	TITLE = {Markov chains and mixing times},
	NOTE = {Second edition of [ MR2466937],
	With contributions by Elizabeth L. Wilmer,
	With a chapter on ``Coupling from the past'' by James G. Propp
	and David B. Wilson},
	PUBLISHER = {American Mathematical Society, Providence, RI},
	YEAR = {2017},
	PAGES = {xvi+447},
	ISBN = {978-1-4704-2962-1},
	MRCLASS = {60J10 (60-01 60B15 60C05 60J27 60K35 68U20 82C22)},
	MRNUMBER = {3726904},
	DOI = {10.1090/mbk/107},
	URL = {https://doi.org/10.1090/mbk/107},
}

@article {gurel2013nonconcentration,
    AUTHOR = {Gurel-Gurevich, Ori and Nachmias, Asaf},
     TITLE = {Nonconcentration of return times},
   JOURNAL = {Ann. Probab.},
  FJOURNAL = {The Annals of Probability},
    VOLUME = {41},
      YEAR = {2013},
    NUMBER = {2},
     PAGES = {848--870},
      ISSN = {0091-1798},
   MRCLASS = {60J10 (05C81 60G40)},
  MRNUMBER = {3077528},
MRREVIEWER = {Elisabetta Candellero},
       DOI = {10.1214/12-AOP785},
       URL = {https://doi.org/10.1214/12-AOP785},
}

@article {norris2017surprise,
    AUTHOR = {Norris, James and Peres, Yuval and Zhai, Alex},
     TITLE = {Surprise probabilities in {M}arkov chains},
   JOURNAL = {Combin. Probab. Comput.},
  FJOURNAL = {Combinatorics, Probability and Computing},
    VOLUME = {26},
      YEAR = {2017},
    NUMBER = {4},
     PAGES = {603--627},
      ISSN = {0963-5483},
   MRCLASS = {60J10},
  MRNUMBER = {3656344},
       DOI = {10.1017/S0963548317000074},
       URL = {https://doi.org/10.1017/S0963548317000074},
}

@article {nadtochiy2019particle,
    AUTHOR = {Nadtochiy, Sergey and Shkolnikov, Mykhaylo},
     TITLE = {Particle systems with singular interaction through hitting
              times: application in systemic risk modeling},
   JOURNAL = {Ann. Appl. Probab.},
  FJOURNAL = {The Annals of Applied Probability},
    VOLUME = {29},
      YEAR = {2019},
    NUMBER = {1},
     PAGES = {89--129},
      ISSN = {1050-5164},
   MRCLASS = {82C22 (35B65 35K20 35R60 60K35 91G80)},
  MRNUMBER = {3910001},
       DOI = {10.1214/18-AAP1403},
       URL = {https://doi.org/10.1214/18-AAP1403},
}

@article {Avrachenkov2018hitting,
    AUTHOR = {Avrachenkov, Konstantin and Piunovskiy, Alexey and Zhang, Yi},
     TITLE = {Hitting times in {M}arkov chains with restart and their
              application to network centrality},
   JOURNAL = {Methodol. Comput. Appl. Probab.},
  FJOURNAL = {Methodology and Computing in Applied Probability},
    VOLUME = {20},
      YEAR = {2018},
    NUMBER = {4},
     PAGES = {1173--1188},
      ISSN = {1387-5841},
   MRCLASS = {60J05 (60J20 90B18)},
  MRNUMBER = {3873621},
       DOI = {10.1007/s11009-017-9600-5},
       URL = {https://doi.org/10.1007/s11009-017-9600-5},
}

@article {Balka2009review,
    AUTHOR = {Balka, Jeremy and Desmond, Anthony F. and McNicholas, Paul D.},
     TITLE = {Review and implementation of cure models based on first
              hitting times for {W}iener processes},
   JOURNAL = {Lifetime Data Anal.},
  FJOURNAL = {Lifetime Data Analysis. An International Journal Devoted to
              Statistical Methods and Applications for Time-to-Event Data},
    VOLUME = {15},
      YEAR = {2009},
    NUMBER = {2},
     PAGES = {147--176},
      ISSN = {1380-7870},
   MRCLASS = {62P10 (60J65)},
  MRNUMBER = {2510862},
       DOI = {10.1007/s10985-008-9108-y},
       URL = {https://doi.org/10.1007/s10985-008-9108-y},
}

\end{document}